\documentclass{article}
\usepackage[latin1]{inputenc}
\usepackage[small]{titlesec}
\usepackage{amsmath}
\usepackage{amsthm}
\usepackage{amssymb}
\usepackage{graphicx}
\usepackage{standalone}
\usepackage{color}
\theoremstyle{plain}
\newtheorem{thm}{Theorem}[section]

\newtheorem{lem}[thm]{Lemma}

\usepackage[small]{titlesec}
\usepackage{blindtext}

\title{The Automorphism Group of $NU(3,q^2)$.}
\author{Federico Romaniello \footnote{Federico Romaniello:
federico.romaniello@unibas.it
Dipartimento di Matematica, Informatica ed Economia -
Universit\`{a} degli Studi della Basilicata - Viale dell'Ateneo Lucano 10
- 85100 Potenza (Italy),}\hfill\newline\hspace*{1.4em}
Valentino Smaldore \footnote{Valentino Smaldore:
valentino.smaldore@unibas.it 
Dipartimento di Matematica, Informatica ed Economia -
Universit\`{a} degli Studi della Basilicata - Viale dell'Ateneo Lucano 10
- 85100 Potenza (Italy).}}

\date{}

\begin{document}
\maketitle
\begin{abstract}
Let $H(n, q^2)$ be a non-degenerate Hermitian variety of $PG(n,q^2)$, $n \geq 2$. Let $NU(n+1,q^2)$ be the graph whose vertices are the points of $PG(n,q^2) \setminus H(n,q^2)$ and two vertices $u,~v$ are adjacent if the line joining $u$ and $v$ is tangent to $H(n, q^2
)$. Then $NU(n + 1, q^2)$ is a strongly regular graph. In this paper we show that the automorphism group of the graph $NU(3,q^2)$ is isomorphic either to $P\Gamma U(3,q)$, the automorphism group of the projective unitary group  $PGU(3,q)$, or to $S_{3} \wr S_4$, according as $q \neq 2$, or $q=2$.

\end{abstract}

\section{Introduction} A \textit{strongly regular graph} with parameters $(v,k,\lambda,\mu)$ is a graph with $v$ vertices where each vertex is incident with $k$ edges, any two adjacent vertices have $\lambda$ common neighbours, and any two non-adjacent vertices have $\mu$ common neighbours.
 Strongly regular graphs were introduced by R. C. Bose in \cite{Bose} in 1963, and ever since it has intensively been investigated. In particular, the eigenvalues of the adjacency matrix of a strongly regular are known; see \cite{Brouwer}:
     a strongly regular graph $G$ with parameters $(v,k,\lambda,\mu)$ has exactly three eigenvalues: $k$, $\theta_{1}$ and $\theta_{2}$ of multiplicity, respectively, $1$, $m_{1}$ and $m_{2}$, where:
    $$\theta_{1}=\frac{1}{2}\big[(\lambda-\mu)+\sqrt{(\lambda-\mu)^{2}+4(k-\mu)}\big],$$
    $$\theta_{2}=\frac{1}{2}\big[(\lambda-\mu)-\sqrt{(\lambda-\mu)^{2}+4(k-\mu)}\big],$$
    $$m_{1}=\frac{1}{2}\Big[(v-1)-\frac{2k-(v-1)(\lambda-\mu)}{\sqrt{(\lambda-\mu)^{2}+4(k-\mu)}}\Big],$$
    $$m_{2}=\frac{1}{2}\Big[(v-1)+\frac{2k-(v-1)(\lambda-\mu)}{\sqrt{(\lambda-\mu)^{2}+4(k-\mu)}}\Big].$$
   The \emph{spectrum} of a strongly regular graph is the triple $(k,\theta_{1}^{m_{1}},\theta_{2}^{m_{2}})$. Therefore, two strongly regular graphs with the same parameters are \textit{cospectral}, that is, they have the same spectrum. Two isomorphic graphs are always cospectral, but the converse is not always true.

   Several strongly regular graphs derive from finite polar spaces, in particular from collinearity graphs or incidence graphs. 
   The \textit{tangent graph} of a polar space $\mathcal{P}$ embedded in a projective space is defined to be the graph in which the vertices are the non-isotropic points with respect to the polarity defining $\mathcal{P}$, and two vertices are adjacent if they lies on a line tangent to $\mathcal{P}$. In \cite{Brouwer} it is proved that tangent graphs are strongly regular.

   The tangent graph $NU(n+1,q^2)$ of the classical Hermitian polar space associated to a non-degenerate Hermitian variety in $PG(n,q^{2})$ was studied by Ihringer, Pavese and Smaldore \cite{Pavese} who pointed out that $NU(n+1,q^2)$ is not uniquely determined by its spectrum for $n\neq3 $. In particular, for $n=2$, they constructed a strongly regular graph that is cospectral but non isomorphic to $NU(n+1,q^2)$. The construction relied on the existence of non-classical unitals embedded in $PG(2,q^2)$ for $q>2$. The goal in this paper is to determine the automorphism group of $NU(3,q^{2})$.
   \begin{thm}
  \label{main}	
     Let $G_{2}=Aut(NU(3,q^{2}))$ be the automorphism group of the graph $NU(3,q^{2})$:
     \begin{enumerate}
      \item if $q\neq2$, $G_{2}\cong P\Gamma U(3,q)$, the semilinear collineation group stabilizing the Hermitian variety $H(2,q^{2})$;
           \item if $q=2$, $G_{2}\cong S_{3} \wr S_4 \cong S_3^4 \rtimes S_4$.
      \end{enumerate}
  \end{thm}

   \section{Preliminary results}
 \subsection{Hermitian varieties and classical unitals}
  In design theory, an \textit{unital} $\mathcal{U}$ is defined as a $2-(a^{3}+1,a+1,1)$ block design, $a\geq3$, i.e. a set of $a^{3}+1$ points arranged into blocks of size $a+1$, such that each pair of distinct points is contained in exactly one block. There exist unitals which are embedded in a projective plane of order $a^{2}$. The classical example consists of the points of the Hermitian curve $H(2,q^{2})$ in $PG(2,q^{2})$.

  Our notation is standard, see \cite{Hirschfeld1}. In particular, $PG(n, q^{2})$ stands for the $n$-dimensional projective space over the finite field $F_{q^{2}}$, with homogeneous projective coordinates $(X_{0}, X_{1}, \dots, X_{n})$. A unitary polarity of $PG(n, q^{2})$ is induced by a non-degenerate Hermitian form on the vector space $V(n+1,q^{2})$. A non-degenerate Hermitian variety $H(n, q^{2})$ consists of the isotropic points of a unitary polarity. A non-degenerate Hermitian variety $H(n, q^{2})$ has $\frac{(q^{n+1}+(-1)^n)(q^n-(-1)^n)}{q^2-1}$ points. For $n=2$, a  canonical form of the Hermitian curve $H(2,q^{2})$ is
  $$X_0^{q+1}+X_1^{q+1}+X_2^{q+1}=0,$$
  and the semilinear group stabilizing $H(2,q^{2})$ is isomorphic to the projective unitary group $P\Gamma U(3,q)$, acting on
  the points of the curve as 2-transitive permutation group, as it was shown in \cite{Hirschfeld1}.

  The Hermitian curve $H(2,q^{2})$ has $q^{3}+1$ points and does not contain isotropic lines. Lines of $PG(2,q^{2})$ are called either \textit{tangent} or \textit{secant} when they meet the curve in 1 or $q+1$ points respectively. In design theory terminology, blocks of the classical unital are the sets of $q+1$ collinear points in the unital. The \textit{dual unital} is defined on the dual projective plane $[PG(2,q^{2})]^*$, where the $q^{3}+1$ tangents are considered as points, and the blocks are the pencil of tangent lines from an external point. The automorphism group of the Hermitian unital acts on the points (as well a on the tangents) as $P\Gamma U(3,q)$ in its unique 2-transitive permutation representation.

 \subsection{Background on the graph $NU(n+1,q^{2})$, $n\geq2$}
  In this paper $H=H(n,q^{2})$ denotes a non-degenerate Hermitian variety in $PG(n,q^{2})$ with $n\geq2$. Then $NU(n+1,q^{2})$ is the strongly regular with vertex set $PG(n,q^{2})\setminus H$ where two vertices are adjacent if they lie on the same tangent line to $H$. Set $\varepsilon=(-1)^{n+1}$ and $r=q^{2}-q-1$. Then $NU(n+1,q^{2})$ has the following parameters:
  $v=\frac{q^{n}(q^{n+1}-\varepsilon)}{q+1}$\\
  $k=(q^{n}+\varepsilon)(q^{n-1}-\varepsilon)$\\
  $\lambda=q^{2n-3}(q+1)-\varepsilon q^{n-1}(q-1)-2$\\
  $\mu=q^{n-2}(q+1)(q^{n-1}-\varepsilon)$.\\
  Furthermore, its complementary graph $\overline{NU(n+1,q^{2})}$ is also strongly regular and has parameters:\\
  $v'=\frac{q^{n}(q^{n+1}-\varepsilon)}{q+1}$\\
  $k'=\frac{q^{n-1}r(q^{n}+\varepsilon)}{q+1}$\\
  $\lambda'=\mu'+\varepsilon q^{n-2}r-\varepsilon q^{n-1}$\\
  $\mu'=\frac{q^{n-1}r(q^{n-2}r+\varepsilon)}{q+1}$.\\
  For further details and information about this graph see \cite{brvan}.
  The construction shows the following lemma, since the group $P\Gamma U(n+1,q^2)$ is transitive on the external point of $H(n,q^2)$ and it preserves the adjacency properties of the graph.
  \begin{lem}
   \label{subgr}
   The subgroup $P\Gamma U(n+1,q)$ of $P\Gamma L(n+1,q^2)$ is an automorphism group of $NU(n+1,q^{2})$.
  \end{lem}

  In the planar case, the graph $NU(3,q^{2})$ arises from a Hermitian curve $H(2,q^{2})$. This curve has $q^{3}+1$ isotropic points, so that the vertex set of $NU(3,q^{2})$ has size  $q^{4}+q^{2}+1-(q^{3}+1)=q^{4}-q^{3}+q^{2}$, while the other parameters are $k=(q^2-1)(9q+1)$, $\lambda=2(q^2-1)$ and $\mu=(q+1)^2$ .

  \begin{lem}
   \label{curve}
   If $q\neq2$, $Aut(NU(3,q^{2}))$ sends $q^{2}$ vertices on a tangent line in other $q^{2}$ vertices on a tangent line.
  \end{lem}

  \begin{proof}
   The $q^{2}$ non-isotropic points on a tangent line, form a maximal clique on the graph $NU(3,q^{2})$. From \cite{Bruen} we know that the maximal cliques of $NU(3,q^{2})$ belong to the $q^{2}$ points on a tangent line, or to $q+2$ points on two different line. Moreover, the automorphism group of a graph sends maximal cliques in maximal cliques. If $q\neq2$, i.e. $q^{2}\neq q+2$, the image of a set of $q^{2}$ non-isotropic aligned points, consists on other $q^{2}$ non-isotropic aligned points.
  \end{proof}

 \section{The action of $G_{n}$ on isotropic points}

  To understand how the automorphism group of $NU(n+1,q^{2})$ can act on the isotropic points, it is useful to introduce the graph $\Gamma_{n}$, whose vertex set consists of all points in $PG(n,q^{2})$, and with the same vertex-edge  incidence relation of $NU(n+1,q^{2})$, namely two vertices are adjacent if the points are on the same tangent. $\Gamma_{n}$ has $\frac{q^{2n}-1}{q^{2}-1}$ vertices, and its automorphism group has two orbits on vertices, one is the orbit $\mathcal{O}_{1}$ consisting of the $v=\frac{q^{n}(q^{n+1}-\varepsilon)}{q+1}$ vertices of $NU(n+1,q^{2})$, each vertex of this type having $k=(q^{n}+\varepsilon)(q^{n-1}-\varepsilon)$ neighbours in $NU(n+1,q^{2})$ and others in $H(n,q^{2})$, and the orbit $\mathcal{O}_{2}$ comprises the points of the Hermitian variety $H(n,q^{2})$, each vertex of this type having neighbours in $NU(n+1,q^{2})$ but no neighbour in $H(n,q^{2})$. In other words, we added an extra point, the $(q^{2}+1)$-th, to each of the maximal cliques arising from tangent lines. However, it should be observed that the resulting graph is not strongly regular anymore.
\begin{lem}
   \label{curve2}
   If $q\neq2$, $Aut(\Gamma_{2})\cong G_{2}$.
\end{lem}
  \begin{proof}
   From Lemma \ref{curve}, maximal cliques corresponding to a tangent are fixed setwise under the action of $G_2$.
   The projection
   $$\pi:\,
   \begin{cases}
   Aut(\Gamma_{2})\rightarrow G_{2}\\
   \lambda\mapsto\overline	{\lambda}=\lambda\big|_{NU(3,q^{2})}.
   \end{cases}
   $$
   is surjective since every automorphism $\overline{\lambda}$ in $G_2$ belongs to $\lambda \in Aut(\Gamma_2)$ such that the action on $NU(3,q^2)$ is described by $\overline{\lambda}$ and the images of the $(q^2+1)$-th tangency isotropic points must be adjacent in $\Gamma_2$ to all the images of their $q^2$ neighbours in $NU(3,q^2)$.

   To prove the injectivity of the projection $\pi$, consider $\lambda\in Aut(\Gamma_{2})$ such that $\pi(\lambda)$ is the identity on $NU(3,q^{2})$.
   When all the points of $NU(3,q^{2})$ are fixed, the tangency point is fixed as well, since it is adjacent in $\Gamma_2$ to each of the $q^{2}$ non-isotropic points on a tangent.

  Therefore $\lambda$ is the identity on $\Gamma_{2}$, thus the kernel of the projection $\pi$ is trivial, and $\pi$ is a monomorphism.
  \end{proof}

  \begin{lem}
   The automorphism group $G_{2}$ of $NU(3,q^{2})$, $q\neq2$, acts 2-transitively on tangents lines.
  \end{lem}

  \begin{proof}
   Since the action of $P\Gamma U(3,q)$ is $2$-transitive on the set of isotropic points of the Hermitian curve, the assertion follows from Lemma \ref{subgr} and Lemma \ref{curve2} taking into account the fact that every isotropic point is the tangency point of a unique tangent to $H(2,q^{2})$.
  \end{proof}

  \begin{lem}
  \label{no3tr}
   $G_{2}$ does not act $3$-transitively on the set of tangents of $NU(3,q^2)$.	
  \end{lem}
  \begin{proof}
   In $PG(2,q^2)$ let us consider the two different configurations of $3$-tangent lines: three concurrent tangents or three tangents which intersect pairwise in three points $P_1,~P_2,~P_3$ on a self-polar triangle.

   In the graph $NU(3,q^2)$ the former set corresponds to three maximal cliques intersecting in their common vertex, whereas in the latter it corresponds to three maximal cliques intersecting pairwise in three different vertices. It is clear that any automorphism of the graph cannot map the former set to the latter.
  \end{proof}

\section{Proof of Theorem \ref{main}}
\subsection{The case $q > 2$}
From now on, assume that $q>2$.
Let $M$ be the minimal normal subgroup of $G_2$ (it is unique by the $2$-transitivity of $G_{2}$). Observe that $M$ cannot be elementary abelian since its order is not a power of a prime (as it contains $P\Gamma U(3,q)$ as its subgroup), so $M$ must be a simple group.

We show that among the finite simple groups which are minimal normal subgroups of $2$-transitive groups, only a few can actually occur in our case.  For this purpose we refer to Cameron's list \cite{Cameron} reported in Table \ref{sgroups}.

\begin{table}[h]
\begin{center}
\begin{tabular}{|c|c|c|}
\hline
$M$         & Degree                  & Remarks         \\ \hline
$A_d$       & $d$                  & $d \geq 5$      \\ \hline
$PSL(d,u)$  & $\frac{u^d-1}{u-1}$  & $d \geq 2$      \\ \hline
$PSU(3,u)$  & $u^3+1$              & $u >2$          \\ \hline
$Sz(u)$     & $u^2+1$              & $u=2^{2a+1} >2$ \\ \hline
$Ree(u)$    & $u^3+1$              & $u=3^{2a+1} >3$ \\ \hline
$PSp(2d,2)$ & $2^{2d-1} + 2^{d-1}$ & $d >2$          \\ \hline
$PSp(2d,2)$ & $2^{2d-1} - 2^{d-1}$ & $d >2$          \\ \hline
$PSL(2,11)$ & $11$                 &                 \\ \hline
$PSL(2,8)$  & $28$                 &                 \\ \hline
$A_7$       & $15$                 &                 \\ \hline
$M_{11}$    & $11$                 &                 \\ \hline
$M_{11}$    & $12$                 &                 \\ \hline
$M_{12}$    & $12$                 &                 \\ \hline
$M_{22}$    & $22$                 &                 \\ \hline
$M_{23}$    & $23$                 &                 \\ \hline
$M_{24}$    & $24$                 &                 \\ \hline
$HS$        & $176$                &                 \\ \hline
$Co_3$      & $276$                &                 \\ \hline
\end{tabular}
\caption{The simple groups $M$ which can occur as minimal normal subgroups of $2$-transitive groups of degree $d$}
\label{sgroups}
\end{center}
\end{table}


The Alternating Group $A_d$ is disregarded from Lemma \ref{no3tr}, as it acts $3$-transitively on the sets of tangents.

It is easily seen that $G_2$ is not any of the sporadic groups in Table \ref{sgroups} since their degree is different from $q^3+1$.

Moreover, for the same reason, $PSL(d,u),~ d > 2$ has also to be dismissed by $\frac{u^d-1}{u-1} \neq q^3+1$. In fact, if $\frac{u^d-1}{u-1} = q^3+1$ were true we would have
$$ u^{d-1}+ u^{d-2}+ \dots + u = q^3.$$
Observe that $u$ and $q$ must be powers of the same prime $p$, and this would lead to a contradiction because the right side of the equation is still a power of $p$ whereas
$$u^{d-1}+ u^{d-2}+ \dots + u = u(u^{d-2} +u^{d-3}+ \dots + 1) \neq p^k, $$ for any $k \in \mathbb{Z}$.
Finally, when $d=2$, a $PSL(2,u)$ is not a subgroup of a $PSU(3,q)$, and this dismiss $PSL(2,8)$.

Before analysing the remaining candidates, take a normal subgroup $M$ of  $G_2$.
Then $G_2$ acts on $M$ in its natural way:
$$\varPhi_{g}: \begin{cases}
M \rightarrow M\\
m \mapsto m^{g}:= g^{-1}mg,
\end{cases}$$
for all $g \in G_2.$\\

\begin{lem}
\label{phi}
$\varPhi_{g}=id \Leftrightarrow g =1$	
\end{lem}
\begin{proof}
It is enough to show that $	\varPhi_{g}=id \Rightarrow g =1$. From $\varPhi_{g}=id$ it follows $g \in C_{G_2}(M)$ where $C_{G_2}(M)$ is the centraliser of $G_2$ in $M$. Suppose on the contrary that $C_{G_2}(M)$ is not trivial. Then it must be primitive since $M$ is $2$-transitive. Let $t$ be a tangent line, and consider $c \in C_{G_2}(M)$. Then, there exists $\alpha \in G_2$ which fixes only $t$. In our case this implies that $t$ is also fixed by $c$, that is, $$ c(\alpha(t))=c(t),\quad \alpha(c(t))= c(\alpha(t))=c(t).$$
Therefore, every tangent line is fixed by $c$, and this implies that $c$ fixes every point as tangent lines are the maximal cliques, so
$C_{G_2}(M)$ is trivial, a contradiction.
\end{proof}

Lemma \ref{phi} shows that $G_2$ is isomorphic to a subgroup of its automorphism group. 

$Sz(u),~Ree(u),~PSp(2d,2)$ cannot be the minimal normal subgroup of $G_2$ as their automorphism group does not contain $P\Gamma U(3,q)$, we refer the reader to \cite{Wilson} for all the details.	

Our discussion on the groups in Table \ref{sgroups} together with lemma \ref{phi} yield that $M=PSU(3,q)$. Since $G_2$ is isomorphic to a subgroup of $Aut(M) \cong P\Gamma U(3,q)$ it follows from Lemma \ref{subgr} that $G_2 \cong P\Gamma U(3,q)$.

 \subsection{The case $q=2$}
  In the smallest case $q=2$, consider the Hermitian curve $H(2,4)$ with $q^{3}+1=9$ isotropic points. The graph $NU(3,4)$ has $q^{4}+q^{2}+1-q^{3}+1=12$ vertices. Through an external point $P\in PG(2,4)\setminus H(2,4)$ there are $q+1=3$ tangent lines, and each line is incident with as many as $q^{2}-1=3$ non-isotropic points other than $P$. Hence the graph is 9-regular. Now let $P$ and $Q$ be two non-adjacent vertices, then $PQ$ is a $(q+1)$-secant line. Through $P$ there are $q+1$ tangent lines, and the same number through $Q$, and they meet each other in $(q+1)^{2}$ other common neighbours, i.e. $\mu=9$. Now let $P$ and $Q$ be two adjacent vertices. Then $PQ$ is tangent to $H$ at $T$. On $PQ$ there are $q^{2}-2$ non-isotropic points other than $P$ and $Q$, so that they have at least $q^{2}-2$ common neighbours. Moreover, through $P$ there are $q$ tangents other than $PQ$, and the same number through $Q$, and they meet each other in $q^{2}$ other common neighbours, i.e. $\lambda=6$. The complementary graph $\overline{NU(3,4)}$ has parameters $(v',k',\lambda',\mu')=(12,2,1,0)$. Therefore it is a trivial strongly regular graph with four connected components isomorphic to the complete graph $K_{3}$. Since $Aut(NU(3,4))=Aut(\overline{NU(3,4)})$, it is enough to find the automorphism group of the complementary graph. Observe that $Aut(\overline{NU(3,4)})$ is the wreath product of four copies of the automorphism group of $K_{3}$, which is the dihedral group $D_{3}\cong S_{3}$, by the Symmetric group of degree $4$ acting as a permutation group on the four connected components:
 $$G_{2}=Aut(NU(3,4))\cong S_{3} \wr S_{4} \cong S_3^4 \rtimes S_4.$$
  In particular, $|G_{2}|=31104.$
  It should be noted that in the graph $\Gamma_{2}$ we make distinction between the two idempotent kinds of maximal cliques, whether we add or do not the fifth point on the Hermitian curve $H(2,4)$, i.e. $G_{2}\ncong Aut(\Gamma_{2})\cong P\Gamma U(3,2)$.

  \section{Conclusion}
  To prove our main theorem, we used the strong hypothesis of the $2$-transitivity of the automorphism group. When the dimension $n$ increases, this fundamental condition is no longer satisfied. With a computer aided search we noted that $P\Gamma U(n+1,q)$ may still be the automorphism group of the graph $NU(n+1,q^2)$, and this will be investigated in the future.

  Moreover, it is known that cospectral strongly regular graphs arising from non-classical unitals are not always isomorphic to $NU(3,q^2)$. It would be interesting to determine their automorphism groups.

\end{document}